\documentclass{amsart}

\usepackage{enumerate}
\usepackage{graphicx,graphics}
\usepackage{amsfonts}
\usepackage{amssymb}
\usepackage{amsthm}
\usepackage{amsmath}
\usepackage{mathrsfs}
\usepackage{hyperref}
\input{xy}
\xyoption{all}

\newtheorem{theorem}{Theorem}[section]
\newtheorem{lemma}[theorem]{Lemma}

\newtheorem{question}[theorem]{Question}

\newtheorem{corollary}[theorem]{Corollary}

\newtheorem{example}[theorem]{Example}

\DeclareMathOperator{\conv}{conv}

\numberwithin{equation}{section}

\begin{document}
\title{Equivariant absolute extensor property on hyperspaces of convex sets}
\author{Natalia Jonard-P\'erez}


\address{Departamento de  Matem\'aticas, Universidad de Murcia, 30100 Espinardo, Murcia, Espa\~na.}

\email{(N.\,Jonard-P\'erez) natalia.jonard@um.es, nat@ciencias.unam.mx}




\thanks{The author was supported by CONACYT (Mexico) grant 204028}

\begin{abstract} 
Let $G$ be a compact group acting on a Banach space $L$  by means of affine transformations.  The action of $G$ on $L$ induces a natural continuous action
on $cc(L)$, the hyperspace of all compact convex subsets of $L$ endowed with the Hausdorff metric topology. The main result of this paper states  that the  $G$-space $cc(L)$ is a $G$-$\rm{AE}$.  Under some extra assumptions, this result can be extended to  $CB(L)$, the hyperspace of all closed and bounded convex subsets of $L$.
\end{abstract}

\maketitle\markboth{ N. JONARD-P\'EREZ}{EQUIVARIANT ABSOLUT EXTENSOR PROPERTY ON HYPERSPACES}

\section{Introduction}

For every Banach space $(L, \|\cdot\|)$ and every subset $M$ of $L$, let us denote by $CB(M)$ the hyperspace of all closed and bounded convex subsets of $M$ endowed with the Hausdorff metric 
$$d_H(A,B)=\inf\{\varepsilon>0\mid A\subset N(B,\varepsilon),~B\subset N(A,\varepsilon)\}$$
where $d$ is the metric induced by the norm and $N(A,\varepsilon)=\{x\in L\mid d(x,A)< \varepsilon\}$. By  $cc(M)$ we denote the subspace of $CB(M)$ consisting of all compact convex sets of $M$. 

The absolute extensor property on hyperspaces of convex sets has been long investigated. 
For a Banach space $L$, it is well known that hyperspaces $cc(L)$ and $CB(L)$ are absolute extensors, while the hyperspace $\conv(L)$ of all closed convex subsets of $L$ (equipped with the Hausdorff metric topology) is an absolute neighborhood extensor (see, \cite{Sakai 2006} and \cite{Sakai 2008}).  

Parallel to the classic theory of absolute extensors,  the notion of an equivariant absolute extensor  ($G$-$\rm{AE}$) and equivariant absolute neighborhood extensor ($G$-$\rm{ANE}$)  has been widely studied  and nowadays there are some very interesting results that generalize the classical Dugundji's Extension Theorem in the equivariant case (see Theorems~\ref{Equivariant Dugundji Theorem} and ~\ref{Equivariant Dugundji Theorem Lie} cf. \cite{A80} and \cite{AntonyanDugundjistheorem}). 
 
Concerning the equivariant absolute extension property on $G$-hyperspaces of compact sets,  in \cite[Proposition~3.1]{Antonyan 2003} S. Antonyan proved that the hyperspace of all compact subsets of a metrizable $G$-space $X$  is a $G$-$\rm{ANE}$ ($G$-$\rm{AE}$) provided that $G$ is a compact group and $X$ is  locally continuum-connected (resp., connected and locally continuum-connected). Also,   
in \cite[Corollary 4.6]{Antonyan 2005}, it was proved that $cc(M)$ is a $G$-$\rm{ANE}$  ($G$-$\rm{AE}$) if $G$ is a Lie group acting linearly on a normed space $L$ and $M\subset L$ is an invariant convex (complete) subset.  

Orbit spaces of hyperspaces of convex sets have been studied in the past because of their relation to such classical objects such as the Hilbert cube and the Banach-Mazur compacta $BM(n)$, $n\ge 2$ (see e.g. \cite{an:00}, \cite{Antonyan 2003}, \cite{AntJon} and \cite{Ant Jon Jua}). In the proof of these results, an important step has been to establish whether or not a certain  hyperspace of convex compact sets of a Banach space is an equivariant absolute extensor. 
 
Motivated by these results, we investigate the equivariant extensor property in hyperspaces of compact and convex subsets of a Banach space $L$ where a compact group acts by means of linear isometries. We also investigate the possibility of extending this result to  $CB(L)$. However, in this case the induced action of $G$ on $CB(L)$ is not always continuous (see Example~\ref{e:accion discontinua}), although in certain cases, for example if the topology on $G$ is the one induced by the norm operator, the induced action on $CB(L)$ is continuous and $CB(L)$ is a $G$-$\rm{AE}$ (see Theorem~\ref{c:main} and Corollary \ref{c:main 2}). Finally, in Theorem~\ref{t:accion afin}, we prove that $cc(L)$ and some invariant subspaces of it are $G$-$\rm{AE}$ if the group $G$ is acting on $L$ by means of affine transformations.

The author wishes to express her gratitude to Sergey Antonyan for suggesting the problem studied in this article.

\section{Preliminaries}

We refer the reader to the monographs \cite{Bredon} and \cite{Palais} for the basic notions of the theory of $G$-spaces. However, we'll recall here some special definitions and results that will be used throughout the paper.

All maps between topological spaces are assumed to be continuous. A map $f:X\to Y$ between $G$-spaces is called $G$-\textit{equivariant} (or simply \textit{equivariant}) if $f(gx)=gf(x)$ for every $x\in X$ and $g\in G$. If $G$ acts trivially on $Y$ (i.e., $gy=y$ for every $g\in G$ and $y\in Y$), an equivariant map $f:X\to Y$ is simply called \textit{invariant}. 

Let $(X,d)$ be a metric $G$-space. If $d(gx,gy)=d(x,y)$ for every $x,y\in X$ and $g\in G$, then we say that $d$ is a $G$-\textit{invariant} metric. That is, every $g\in G$ is actually an isometry of $X$ with respect to the metric $d$. We also say that $G$ acts \textit{isometrically} on $X$.

A point $x_0$ in a $G$-space $X$ is called a $G$-\textit{fixed point} if $gx_0=x_0$ for every $g\in G$. We say that $A\subset X$ is  \textit{$G$-invariant}  (or simply \textit{invariant}) if $ga\in A$ for every $a\in A$ and $g\in G$.

Let $G$ be a topological group and $X$ a (real) linear space.
 We call $X$ a \textit{linear $G$-space} if there is a \textit{linear action} of $G$ on $X$, i.e., if 
\begin{equation*}\label{linact}
g(\alpha x+\beta y)=\alpha(gx)+\beta(gy)
\end{equation*}
for every $g\in G$, $\alpha,\beta\in\mathbb{R}$ and $x,y\in X$. If, in addition, $X$ is a Banach space and the  norm is $G$-invariant,
we will say that $X$ is a \textit{Banach $G$-space}. 
On the other hand, we say that an action of $G$ on $X$ is \textit{affine} if $g(tx+(1-t)y)=tgx+(1-t)gy$ for every $x,y\in X$ and $t\in [0,1]$. Obviously, every linear action is also an affine action.

If $L$ is a Banach space and $G$ is a compact group, we denote by $C(G,L)$ the  space of all continuous maps from $G$ into $L$ equipped with the compact open topology which, due to the compactness of $G$, can be generated  by the norm
\begin{equation}\label{f:norma en C(G,L)}
\|f\|=\sup_{x\in G}\{\|f(x)\|\},\quad f\in C(G, L)
\end{equation}

We  consider the action of $G$ on $(C(G,L), \|\cdot\|)$ defined by the rule
\begin{equation}\label{f:accion en C(G,L)}
gf(x)=f(xg)\quad x,g\in G, ~~f\in C(G,L).
\end{equation}
This action is continuous, linear, and the norm $\|\cdot\|$ becomes a $G$-invariant norm (see \cite{A80} cf. \cite{Smirnov}). Thus $(C(G,L), \|\cdot\|)$ equipped with the action~(\ref{f:accion en C(G,L)}) is a Banach $G$-space.

For a given topological group $G$, a metrizable $G$-space $X$ is called a $G$-\textit{equivariant absolute neighborhood extensor}
 (denoted by $X\in G$-$\mathrm{ANE}$) if for any  metrizable $G$-space $Z$ and any equivariant map $f:A\to X$ from an invariant closed subset $A\subset Z$ into $X$, there exists an  invariant neighborhood $U$ of $A$ in $Z$ and an equivariant map $F:U\to X$ such that $F|_A=f$.  If we can always take $U=Z$, then we say that  $X$ is a $G$-\textit{equivariant absolute extensor} (denoted by $X\in G$-$\mathrm{AE}$).
 
The following two theorems will be the key in the proof of our main result.

\begin{theorem}[{\rm \cite[Theorem~2]{A80}}]\label{Equivariant Dugundji Theorem}
Let $G$ be a compact group acting linearly on a locally convex metric linear space $X$ and $K$ an invariant complete convex subset of $X$. Then $K$ is a $G$-{\rm AE}. 
\end{theorem}

\begin{theorem}[\cite{AntonyanDugundjistheorem}]\label{Equivariant Dugundji Theorem Lie} Let $G$ be a compact Lie group and  $X$ be a locally convex linear $G$-space. Then every convex invariant subset $K\subset X$ is a $G$-{\rm ANE}. Furthermore, if $K$ has a $G$-fixed point, then $K$ is a $G$-$\rm{AE}$.
\end{theorem}


For any subsets $A$ and $B$ of a linear space $L$  and $t\in\mathbb{R}$, the sets $$A+B=\{a+b\mid a\in A,~ b\in B\}\quad\quad\textnormal{and}\quad\quad tA=\{ta\mid a\in A\}$$ are called the \textit{Minkowski sum} of $A$ and $B$ and the \textit{product} of $A$ by $t$, respectively. It is well known that these operations preserve compactness and convexity.
However, if $A$ and $B$ are closed subsets, it is not always true that $A+B$ is closed.

The \textit{Hausdorff distance} between two arbitrary  subsets $A$ and $B$ of a metric space $(X,d)$ is defined by the rule:
\begin{equation}\label{Hausdorff}
d_H(A,B)=\inf\{\varepsilon >0\mid A\subset N(B, \varepsilon) \text{ and }B\subset N(A,\varepsilon)\},\quad A,B \subset X.
\end{equation}
where $N(A,\varepsilon)=\{x\in L\mid d(x,A)< \varepsilon\}$.
It is well-known that the Hausdorff distance satisfies 
\begin{equation}\label{e:distancia entre cerraduras}
d_H(A,B)=d_H\big(\overline{A},\overline{B}\big)
\end{equation} 
for every pair $A$ and $B$ of bounded subsets of $X$.
Additionally, if $X$ is a linear normed space and $A$ and $B$ are convex subsets, then
\begin{equation}\label{e:d hausdorff invariante bajo traslaciones}
d_H(A+C,B+C)=d_H(A,B)
\end{equation}
for every bounded set $C\subset X$. For these and other properties consult \cite{Schmidt} (c.f. \cite{Radstrom}).

Finally, we recall the following well known result which will be used in the last part of this work.

\begin{lemma}\label{l:closed hyperspace}
Let $M$ be a closed subset of a metric space $X$. Then $2^{M}$ is closed in $2^{X}$, where $2^{X}$ denotes the hyperspace of all nonempty compact subsets of $X$ equipped with the Hausdorff metric. In particular, if $X$ is a linear space, $cc(M)$ is closed in $cc(L)$.
\end{lemma}
\begin{proof} It is enough to prove that $2^{X}\setminus 2^{M}$ is open.
For any $A\in 2^{X}\setminus 2^{M}$, there exist $a\in A\setminus M$ and $\varepsilon >0$, such that $d(a, M)>\varepsilon$. From this inequality, we get that  no $C\subset M$ satisfies $d_H(C, A)<\varepsilon$ and therefore $2^{X}\setminus 2^{M}$ is open, as required.
\end{proof}

\section{Group actions on hyperspaces of convex sets}

Let $L$ be a Banach space and $G$ a compact group acting continuously on $L$ by means of affine homeomorphisms.
 The action of $G$ induces a natural action on $CB(L)$ by the rule
\begin{equation}\label{f:accion}
(g, A)\longrightarrow gA:=\{ga\mid a\in A\}.
\end{equation}

It is easy to verify that the restriction of this action to $cc(L)$ is always continuous. However, as we will show in the following example, the action (\ref{f:accion}) is not always  continuous on $CB(L)$ even if $L$ is a Banach $G$-space.

\begin{example}\label{e:accion discontinua}
Let $G=\mathbb Z_{2}^{\infty}$ denote the Cantor group. Every element $x\in G$ can be represented as a sequence $x=(x_i)_{i\in\mathbb N}$ where $x_i\in\mathbb Z_2=\{1,-1\}$. In this case the product topology on $G$ is a metrizable group topology if the operation is defined  by the rule
$$xy=(x_iy_i)_{i\in\mathbb N}\quad\text{ for each }~x=(x_i)_{i\in\mathbb N}\in G,~~y=(y_i)_{i\in\mathbb N}\in G.$$
Now, consider the space $C(G,\mathbb R)$ of all continuous real-valued maps defined on $G$. The space $L=C(G, \mathbb R)$ equipped with the  norm {\rm(\ref{f:norma en C(G,L)})} and  the action {\rm(\ref{f:accion en C(G,L)})}  becomes a Banach $G$-space with the property that the action defined in {\rm(\ref{f:accion})} is not continuous on $CB(L)$.
\end{example}

\begin{proof}
Consider the set $A\subset L$ consisting of all continuous maps $f:G\to [0,1]$ such that $f(e)=0$ where $e\in G$ is the identity element $e=(1,1,\dots).$
Obviously $A$ is a closed and  bounded convex subset of $L$.
We will show that the action of $G$ on $CB(L)$ is not continuous in the pair $(e, A)$. Consider any  neighbourhood $Q$ of $e$ and pick an arbitrary point $y\in Q\setminus\{e\}$.  
By Urysohn's Lemma there exists a continuous map $f:G\to [0,1]$ such that $f(e)=0$ and $f(y)=1$. Evidently $f\in A$. Furthermore, for every $\varphi\in A$, 
$$\|\varphi-yf\|=\sup_{x\in G} |\varphi(x)-yf(x)|\geq |\varphi(e)-yf(e)|=|\varphi(e)-f(y)|=|0-1|=1.$$ 

This directly implies that $d_H(A, yA)\geq 1$ and therefore the action defined in (\ref{f:accion}) cannot be continuous on $CB(L)$.

\end{proof}

Despite the previous example,  certain cases exist where the action (\ref{f:accion}) is continuous on $CB(L)$. The most simple example is the case when $G$ is finite. A more interesting example is explained below.

First we'll state the following lemma, which will be used several times in this paper. The proof will be omitted, as it is a direct consequence of the Hausdorff distance definition. 

\begin{lemma}\label{l:dhinvariante}
Suppose that $G$ acts linearly and isometrically on a normed linear space $L$. Then, the Hausdorff metric on $CB(L)$ induced by the $G$-invariant norm is $G$-invariant if we equip $CB(L)$ with the (non-necessarily continuous) action defined in \rm{(\ref{f:accion})}.
\end{lemma}

For every  Banach $G$-space $L$, the group $G$ is in fact a subgroup of $U(L)=\{T:L\to L\mid T \text{ is a linear operator with } \|T\|_*=1\}$, where $\|\cdot\|_{*}$ denotes the norm operator:
$$\|T\|_{*}=\sup_{x\in L\setminus\{0\}}\frac{\|T(x)\|}{\|x\|}.$$
  The topology on $G$ induced by the norm operator is the \textit{topology of the uniform convergence on bounded sets}.
  
  \begin{example}
  Let $L$ be a Banach $G$-space where $G$ is a compact group. Suppose that the topology on $G$ contains the topology of the uniform convergence on bounded sets. Then the induced action on $CB(L)$ is continuous. 
  \end{example} 

\begin{proof}
Pick a pair $(g,A)\in G\times CB(L)$ and let $\varepsilon>0$.
Define $M=\max_{a\in A} \|a\|$ and consider $\delta<\varepsilon/2M$.
Denote by $1_L$ the identity map of $L$ and let $O=\{h\in G\mid \|1_L-h\|_*<\delta\}$.
Thus, for any $a\in A$ and $h\in O$ we have:
$$\|a-ha\|\leq \|1_L-h\|_*M<M\delta$$
and hence 
\begin{equation}\label{d:1}
d_H(A,hA)\leq M\delta<\varepsilon/2,\quad \text{for every }h\in O.
\end{equation}
Since the topology of $G$ contains the topology of the uniform convergence, the set $U:=gO=\{gt\mid \|1_L-t\|_*<\delta\}$ is open in $G$. Observe that every $h\in U$  satisfies 
$g^{-1}h\in O$. Let $Q$ be the $\varepsilon/2$-neighborhood around $A$ in $CB(L)$.
Then, if $(h,B)\in U\times Q$, we can use inequality (\ref{d:1}) and Lemma~\ref{l:dhinvariante} below to conclude that
\begin{align*}
d_H(gA, hB)&=d_H(A, g^{-1}hB)\leq d_H(A,g^{-1}hA)+d_H(g^{-1}hA,g^{-1}hB)\\
&=d_H(A,g^{-1}hA)+d_H(A,B)<\varepsilon/2+\varepsilon/2=\varepsilon.  
\end{align*}
 Now the proof is complete.
\end{proof}

%
%

\section{Equivariant embeddings of hyperspaces}

The main purpose of this section is to reconstruct the R\aa dstr\"om-Schmidt Embedding Theorem (\cite{Radstrom} and \cite{Schmidt}) in order to prove that the hyperspaces $cc(L)$ (and, in some cases, $CB(L)$) can be  embedded as an invariant closed convex subset of a Banach $G$-space.

In what follows,  $L$ will always denote a   Banach $G$-space.
Also, we will use the symbol $\mathcal K$ to denote simultaneously the 
hyperspace $CB(L)$ or $cc(L)$.

Let us denote by $H(\mathcal K)$  the quotient space of $\mathcal K\times \mathcal K/\sim$ obtained by the following equivalence relationship:
$$(A,B)\sim (C,D)\Longleftrightarrow \overline{A+D}=\overline{B+C}$$

For every $(A,B)\in \mathcal K\times \mathcal K $ we denote by $\langle A,B\rangle$ its corresponding equivalence class in $\mathcal H(\mathcal K)$. 

The space $\mathcal H(\mathcal K)$ becomes a real linear space if we define the sum by:
$$\langle A,B\rangle+\langle C,D\rangle:=\langle \overline {A+C}, \overline{C+D}\rangle$$
and the scalar multiplication by:
\begin{equation*}
t\langle A,B\rangle:=\begin{cases}
\langle tA,tB\rangle& t\geq 0,\\
\langle-tB, -tA\rangle& t\leq 0.
\end{cases}
\end{equation*}

The class $\langle \{0\}, \{ 0\} \rangle$  corresponds to the origin of $\mathcal H(\mathcal K)$ and the inverse of the element $\langle A,B\rangle\in \mathcal H(\mathcal K)$ coincides with the class $\langle B,A \rangle$.
Additionally,  $\mathcal H(\mathcal K)$ becomes a normed space if we define the following norm:
$$\|\langle A, B\rangle\|:=d_H(A,B),$$
where $d_H$ is the Hausdorff metric on $\mathcal K$ induced by the norm  of $L$.

Now,  the map $j:\mathcal K\to \mathcal H(\mathcal K)$ defined by
$$j(A):=\langle A, \{0\}\rangle$$
is an isometric embedding satisfying the following conditions:
\begin{enumerate}[(a)]
\item $j(tA)=tj(A)$
\item $j(\overline{A+B})=j(A)+j(B)$
\end{enumerate}
for every $t\geq 0$ and $A, B\in\mathcal K$.
Details about this construction can be consulted in \cite{Radstrom} and \cite{Schmidt}.

Now, suppose that $L$ is a Banach $G$-space. In this case, we can define a natural action on $\mathcal H(\mathcal K)$ by the rule

\begin{equation}\label{f:accion en H(K)}
g\langle A,B\rangle=\langle gA, gB\rangle.
\end{equation}

\begin{theorem}\label{t:equivariant embedding} Let $L$ be a Banach $G$-space and suppose that the induced action on $\mathcal K$ is continuous (if $\mathcal K=cc(L)$, this is always true). Then $\mathcal H(\mathcal K)$ equipped with the action defined in {\rm(\ref{f:accion en H(K)})} is a Banach $G$-space, $j(\mathcal K)$ is closed in $\mathcal H(\mathcal K)$  and the embedding $j$ is $G$-equivariant and isometric. 
\end{theorem}

\begin{proof}
To see that $\mathcal H(\mathcal K)$ is a Banach space, $j$ is an isometry and $j(\mathcal K)$ is closed, the reader can consult \cite{Schmidt} and \cite{Radstrom}. We only prove here the facts concerning the action of $G$. Namely, we will prove that the action is well defined, continuous, isometric and that the embedding $j$ is equivariant.
Indeed, if $\langle A,B\rangle =\langle C,D\rangle$, then 
$$\overline{A+D}=\overline{B+C}.$$
Since every $g\in G$ is a linear homeomorphism, we have that
\begin{align*}
\overline{gA+gD}&=\overline{g(A+D)}=g(\overline{A+D})\\
&=g(\overline{B+C})=\overline{g(B+C)}=\overline{gB+gC}
\end{align*}
and so, the action is well defined.
Proving that this action is linear is simple routine and we leave the details to the reader. 
By Lemma~\ref{l:dhinvariante} we have
$$\|g\langle A,B\rangle\|=d_H(gA,gB)=d_H(A,B)=\|\langle A,B\rangle\| $$
which implies that each $g\in G$ is an isometry on $\mathcal H(\mathcal K)$. 
Furthermore, for each $A\in K$ and $g\in G$, we have
$$j(gA)=\langle gA, \{0\}\rangle=\langle gA, g\{0\}\rangle=g\langle A, \{0\}\rangle=gj(A).$$
This last equality means that the embedding is equivariant.

Finally, it rests  to prove that this action is continuous.
Let $(g, \langle A, B\rangle)\in G\times\mathcal H(\mathcal K)$ where $A$ and $B$ are fixed elements of $\mathcal K$ representing the equivalence class $\langle A,B \rangle$. Since the action of $G$ on $\mathcal K$ is continuous,  for each $\varepsilon>0$ it is possible to find a symmetric neighbourhood $O$ of the identity element $e$ in $G$ such that $d_H(A, hA)<\varepsilon/4$ and $d_H(B, hB)<\varepsilon/4$ for every $h\in O$.

Let $U=gO$ and let $V$ be the $\varepsilon/2$-ball around $\langle A, B\rangle $ in $\mathcal H(\mathcal K)$. Since $O$ is symmetric, it happens that $h^{-1}g\in O$ for every $h\in gO$. 
Now,  it follows from the definition of the norm on $\mathcal H(\mathcal K)$ and properties (\ref{e:distancia entre cerraduras}) and (\ref{e:d hausdorff invariante bajo traslaciones}) that 
\begin{align*}
\|g\langle A, B \rangle-h\langle C, D\rangle\|=\|h^{-1}g\langle A, B \rangle+\langle D, C\rangle\|=\|\langle \overline{h^{-1}gA+D},\overline{h^{-1}gB+C}\rangle\|\\
=d_H(\overline{h^{-1}gA+D},\overline{h^{-1}gB+C})=d_H(h^{-1}gA+D,h^{-1}gB+C)\\
\leq d_H(h^{-1}gA+D, A+D)+d_H(A+D, B+C)+d_H(B+C, h^{-1}gB+C)\\
=d_H(h^{-1}gA, A)+d_H(\overline{A+D}, \overline{B+C})+d_H(B, h^{-1}gB)\\
=d_H(h^{-1}gA, A)+ \|\langle A,B \rangle-\langle C,D\rangle\|+d_H(B, h^{-1}gB)\\
<\varepsilon/4+\varepsilon/2+\varepsilon/4=\varepsilon.
\end{align*}
 We conclude from the previous inequality that the action of $G$ on $\mathcal H(\mathcal K)$ is continuous and now the proof is complete.
\end{proof}

\section{Equivariant absolute extensor property}

In this section we prove the main results of this paper. Let us remark that if $G$ is a Lie group,  Theorem~\ref{c:main} is a particular case of \cite[Corollary 4.6]{Antonyan 2005}.

\begin{theorem}\label{c:main}
Let $G$ be a compact group and $L$ a Banach $G$-space.
Then, $cc(L)$ is a $G$-$\rm{AE}$. Additionally, if the induced action of $G$ on $CB(L)$ is continuous, then $CB(L)$ is also a $G$-$\rm{AE}$. 
\end{theorem}

\begin{proof}
As previously, let us denote by $\mathcal K$ the hyperspace $cc(L)$ or $CB(L)$ equipped with the induced (continuous) action of $G$.
According to Theorem~\ref{t:equivariant embedding}, $j(\mathcal K)$ is an invariant closed and convex  subset of a Banach $G$-space. 
Now, according to Theorem~\ref{Equivariant Dugundji Theorem}, $j(\mathcal K)$ is a $G$-$\rm{AE}$. Since $j(\mathcal K)$ and $\mathcal K$ are $G$-homeomorphic,  this directly implies both parts of the theorem.
\end{proof}

We will say that a family $\mathcal C\subset \mathcal K$ is convex iff
$$\overline{tA+(1-t)B}\in \mathcal C \quad \text{for every }A, B\in \mathcal C, ~~~~t\in [0,1]. $$

The following are examples of convex families:
\begin{enumerate}[i)]
\item $CB(M)$ and $cc(M)$, where $M$ is a closed convex subset of $L$.
\item  The family of all infinite dimensional convex compacta of $L$. 
\item  The family of all finite dimensional convex compacta of $L$.
\item The family of all infinite dimensional closed and bounded convex subsets of $L$.

\item The family of all convex bodies of $L$ (that is, the family of all closed and bounded convex subsets of $L$ with non empty interior).

\end{enumerate}

\begin{corollary}\label{c:main 2}
Let $G$ be a compact group and $L$ a Banach $G$-space. For any   
convex invariant subset $\mathcal C\subset \mathcal K$, the following statements are true: 
\begin{enumerate}[\rm(1)]
\item $\mathcal C\in\rm{AE}$
\item If $\mathcal C$ is closed in $\mathcal K$ then $\mathcal C\in G$-$\rm{AE}$ 
\item If $G$ is a Lie group, then $\mathcal C\in G$-$\rm{ANE}$.
If, in addition, $\mathcal C$ has a $G$-fixed point, then $\mathcal C\in G$-$\rm{AE}$. 
\end{enumerate}
\end{corollary}

\begin{proof}
By Theorem~\ref{t:equivariant embedding} and  properties (a) and (b), $\mathcal C$ is $G$-isometric to an invariant convex subset of a Banach $G$-space. 

Now, sentence (1) follows immediately  from  the classic Dugundji's Extension Theorem (see, e.g., \cite{Hu}).
To prove (2),  we  simply apply Theorem~\ref{Equivariant Dugundji Theorem}.
Finally, sentence (3) follows directly from Theorem~\ref{Equivariant Dugundji Theorem Lie}. 
\end{proof}



As an application, we obtain an alternative proof for \cite[Corollary 4.8 (1)]{Antonyan 2005}:

\begin{corollary}\label{c: corollary antonyan}
Consider the natural action of the orthogonal group $O(n)$ on the $n$-dimensional euclidean space $\mathbb R^{n}$. The following hyperspaces of convex sets are $O(n)$-$\rm{AE}$ spaces:
\begin{enumerate}[\rm(1)]
\item $cc(\mathbb R^{n})$,
\item $cc(\mathbb B^{n})$, where $\mathbb B^{n}=\{x\in\mathbb R^{n}\mid \|x\|\leq 1\}$ and $\|\cdot\|$ denotes the euclidean norm,
\item the hyperspace of all convex bodies of $\mathbb R^{n}$ (or $\mathbb B^{n})$,
\item the hyperspace of all centrally symmetric convex bodies of $\mathbb R^{n}$.
\end{enumerate}
\end{corollary}

\section{Group actions by means of affinities}

Consider a compact group $G$ acting continuously on a Banach  space $L$ by means of affine transformations. 
  In this section we will show that the $G$-space $cc(L)$ is always a $G$-$\rm{AE}$. The technique used to prove this result is based on an argument used in \cite{Ant Jon Jua}
(cf.   \cite[Theorem~C]{A80} and \cite[Theorem~2]{Smirnov},)

For any Banach space $L$, denote by $\widetilde L$ the Banach $G$-space $C(G,L)$ equipped with the norm (\ref{f:norma en C(G,L)}) and the continuous action described in (\ref{f:accion en C(G,L)}).
Now, if $G$ acts affinely on $L$,   the map $\Phi:L\to \widetilde L$ defined by the rule 
 \begin{equation}\label{f:funcion Phi}
 \Phi(x)(g)=gx
 \end{equation}
 is an equivariant affine embedding (see \cite[Theorem~C]{A80},  \cite[Proposition 3.1]{Ant Jon Jua}   and \cite[Theorem~2]{Smirnov}).
 In addition, since $L$ is a Banach space, $\Phi(L)$ is closed in $\widetilde L$. Indeed, 
 any  convergent sequence $(\Phi(x_n))_{n\in \mathbb N}\rightsquigarrow h\in \widetilde L$ is also a Cauchy sequence on $\widetilde L$. Thus, for each $\varepsilon >0$, we can find  $N\in\mathbb N$ such that 
 $$\sup_{g\in G}\|gx_n-gx_m\|<\varepsilon\quad n,m\geq N.$$
 In particular, $\|x_n-x_m\|<\varepsilon$ for every $n,m\geq N$ and so, $(x_n)_{n\in\mathbb N}$ is also a Cauchy sequence in the Banach space $L$. Then, there exists a point $x\in L$ with the property that $(x_n)_{n\in\mathbb N}\rightsquigarrow x$. The continuity of map $\Phi$ implies that $(\Phi(x_n))_{n\in\mathbb N}\rightsquigarrow \Phi(x)$ and thus $\Phi(x)=h\in \Phi(L)$. This proves that $\Phi(L)$ is closed in $\widetilde L$.

Since $\Phi$ is affine and equivariant, it naturally extends to an equivariant bijective map $\widetilde \Phi:cc(L)\to cc(\Phi(L))$. Since the Hausdorff metric topology on $cc(L)$ depends only on the topology of $L$ (see, \cite{Beer}), this map is in fact an equivariant homeomorphism.

Now, according to Lemma~\ref{l:closed hyperspace}, $cc(\Phi(L))$ is closed in $cc(\widetilde L)$ and therefore $cc(L)$ is $G$-homeomorphic to an invariant closed convex  subset of a Banach $G$-space. Furthermore, any invariant (closed) convex subset $\mathcal C$ of $cc(L)$ is $G$-homeomorphic to an invariant (closed) convex subset of a Banach $G$-space.  All the previous arguments can  now be combined with Corollary~\ref{c:main 2} in order to prove the following theorem:

\begin{theorem}\label{t:accion afin}
Let $L$ be a Banach space and $G$ a compact topological group acting continuously on $L$ by means of affinities. Then, the following sentences are true:
\begin{enumerate}[\rm(1)]
\item Any closed and convex invariant subset $\mathcal C$ of $cc(L)$ is a $G$-$\rm{AE}$. In particular $cc(L)$ is a $G$-$\rm{AE}$.
\item If $G$ is a Lie group, any convex invariant subset $\mathcal C$ of $cc(L)$ is a $G$-$\rm{ANE}$. Additionally, if $\mathcal C$ has a fixed point, then $\mathcal C\in G$-$\rm{AE}$. 
\end{enumerate}
\end{theorem}

\textbf{Final Remark.}
The map $\Phi$ defined in (\ref{f:funcion Phi}) can also be extended to an equivariant biyection from $CB(L)$ to $CB(\Phi(L))$. However,  since the map $\Phi$ is not uniformly continuous, the hyperspaces $CB(L)$ and  $CB(\Phi(L))$ may not be homeomorphic. For this reason, we cannot use the previous arguments to extend Theorem~\ref{t:accion afin} to  $CB(L)$, which leads to the following  question. 

\begin{question}
Suppose that a compact group $G$ acts continuously and affinely on a Banach space $L$ in such a way that the induced action on $CB(L)$ is continuous. Is  $CB(L)\in G$-$\rm{AE}$? 
\end{question}

\bibliographystyle{amsplain}

\end{document}